\documentclass[12pt]{article}
\usepackage[latin2]{inputenc}
\usepackage{hyperref}
\usepackage{amssymb}
\usepackage{paralist}
\usepackage{amsmath}
\usepackage{amsthm}
\usepackage{amsfonts}
\usepackage{graphicx}
\usepackage{wrapfig}
\usepackage{enumerate}
\usepackage{url}
\usepackage{fullpage}
\usepackage{verbatim}

\newcommand{\mc}[1]{\mathcal{#1}}

\DeclareMathOperator{\Cov}{Cov}
\DeclareMathOperator{\cov}{cov}
\DeclareMathOperator{\CAN}{CAN}
\DeclareMathOperator{\ACAN}{ACAN}

\newcommand{\covt}{\cov_t}
\newcommand{\covm}{\cov_{\max}}

\newcommand{\floor}[1]{\left \lfloor #1 \right \rfloor}
\newcommand{\ceil}[1]{\left \lceil #1 \right \rceil}

\newcommand{\card}[1]{\left| #1 \right|}
\renewcommand{\lg}{\log_2}
\newcommand{\eps}{\varepsilon}

\newtheorem{thm}{Theorem}[section]
\newtheorem{lemma}[thm]{Lemma}
\newtheorem{cor}[thm]{Corollary}
\newtheorem{prop}[thm]{Proposition}
\newtheorem{defi}[thm]{Definition}
\newtheorem{ques}[thm]{Question}

\title{A Semi-Random Construction of Small Covering Arrays}
\author{Shagnik Das\thanks{Freie Universit\"at Berlin, Institut f\"ur Mathematik, Arnimallee 3, 14195 Berlin, Germany.} \thanks{E-mail: \texttt{shagnik@mi.fu-berlin.de}} \and Tam\'as M\'esz\'aros\footnotemark[1] \thanks{E-mail: \texttt{tamas.meszaros@fu-berlin.de}. Position funded by the DRS POINT Fellowship Program.}}

\begin{document}

\maketitle

\begin{abstract}
Given a set $S$ of $v \ge 2$ symbols, and integers $k \ge t \ge 2$ and $N \ge 1$, an $N \times k$ array $A \in S^{N \times k}$ is an $(N; t, k, v)$-covering array if all sequences in $S^t$ appear as rows in every $N \times t$ subarray of $A$.  These arrays have a wide variety of applications, driving the search for small covering arrays.  The covering array number, $\CAN(t,k,v)$, is the smallest $N$ for which an $(N; t,k,v)$-covering array exists.

In this paper, we combine probabilistic and linear algebraic constructions to improve the upper bounds on $\CAN(t,k,v)$ by a factor of $\ln v$, showing that for prime powers $v$, $\CAN(t,k,v) \le (1 + o(1)) \left( (t-1) v^t / (2 \lg v - \lg (v+1)) \right)\lg k$, which also offers improvements for large $v$ that are not prime powers.  Our main tool, which may be of independent interest, is a construction of an array with $v^t$ rows that covers the maximum possible number of subsets of size $t$.
\end{abstract}

\section{Introduction} \label{sec:intro}

In the last few decades, a great deal of research has been devoted to the study of orthogonal and covering arrays, an important class of combinatorial designs.  This research is motivated by numerous applications, in particular to computer science and the design of experiments, and one of the major open problems in this area is to determine how small these arrays can be.  In this paper we improve the best-known general bounds for this problem.  We first provide a brief introduction to the subject, surveying some relevant results from the literature, before presenting our new results.

\subsection{Background and previous results} \label{subsec:background}

Let $A$ be an $N \times k$ array, whose entries come from some set $S$ of $v$ symbols; that is, $A \in S^{N \times k}$.  In the context of experimental design, $N$ represents the number of trials to be carried out, $k$ denotes the number of factors to be tested, and $S$ is the set of levels these factors can take.  The objective is to determine how subsets of the factors interact with one another.  To that end, given a set $Q$ of $t$ column indices, we denote by $A_Q$ the $N \times t$ subarray obtained by restricting $A$ to the columns in $Q$.

\begin{defi}[Orthogonal arrays]
Given a set $S$ of $v$ symbols, $k$ columns, an \emph{index} $\lambda \in \mathbb{N}$ and a \emph{strength} $t \in \mathbb{N}$, let $N = \lambda v^t$.  An $N \times k$ array $A \in S^{N \times k}$ is an \emph{$(N;t,k,v)$-orthogonal array} if for every subset $Q$ of $t$ columns, every sequence in $S^t$ appears exactly $\lambda$ times as a row of the subarray $A_Q$.
\end{defi}

An orthogonal array is therefore a very regular structure, behaving uniformly with respect to every subset of $t$ columns, giving rise to an important application in theoretical computer science.  Many randomised algorithms use some large number $k$ of independent random variables, each uniformly distributed over a set $S$ of size $v$, thus using the exponentially large probability space $S^k$.  However, quite often one only requires the weaker condition that the random variables be $t$-wise independent, for some small $t$.  Given an $(N; t, k, v)$-orthogonal array, a uniform distribution on the $N$ rows of this array provides a probability space with the desired independence, and, if $N$ is small, this allows for brute-force derandomisation of the algorithm.

The interest, then, is in determining how few rows an orthogonal array with a given strength and number of columns can have.  In one of the early papers on the subject, Plackett and Burman~\cite{PB46} provided sharp bounds for orthogonal arrays of strength two, showing how the number of rows must grow with the number of columns.

\begin{thm}[Plackett--Burman~\cite{PB46}, 1946] \label{thm:OAstrengthtwo}
If an $(N; 2, k, v)$-orthogonal array exists, we must have
\[ k \le \floor{ \frac{N - 1}{v - 1} }. \]
\end{thm}

Following this initial focus on orthogonal arrays of strength two, Rao~\cite{Rao46, Rao47} generalised the notion to arrays of strength $t$, giving rise to the modern study of orthogonal arrays.  For an account of the last half-century's developments in the field, the reader is referred to the book of Hedayat, Sloane and Stufken~\cite{HSS99}.

However, the high level of regularity required of an orthogonal array places severe restrictions on the possible values of the parameters, and hence the fundamental question asks for which parameters an $(N; t, k, v)$-orthogonal array exists.  For many applications, one is willing to make do with a smaller, yet less regular, construction, giving rise to the relaxation of orthogonal arrays to covering arrays.  Here, one only requires that all sequences in $S^t$ appear at least once in every $N \times t$ subarray, dropping the condition that they appear equally often.  The primary question is now an extremal one --- how small can an array satisfying this weaker condition be?

\begin{defi}[Covering arrays]
Given a set $S$ of $v$ symbols, $N$ rows, $k$ columns and a \emph{strength} $t \in \mathbb{N}$, let $A \in S^{N \times k}$ be an $N \times k$ array.  We say that $A$ \emph{covers} a subset $Q$ of $t$ columns if every sequence in $S^t$ appears at least once as a row of the subarray $A_Q$.

The array $A$ is an \emph{$(N; t, k, v)$-covering array} if it covers every $t$-subset of the $k$ columns, and the \emph{covering array number} $\CAN(t,k,v)$ is the minimum number $N$ of rows for which an $(N; t, k, v)$-covering array exists.
\end{defi}

In the binary case $v = 2$, $S$ may be taken to be the set $\{0, 1\}$, and an $(N; t, k, 2)$-covering array can then be interpreted as the incidence matrix $A$ of a family $\mc F$ of $N$ subsets of the ground set $[k]$.  In this setting, covering is often referred to as \emph{shattering}, as a $t$-subset $Q \subseteq [k]$ is shattered whenever every one of its $2^t$ subsets appears as an intersection of $Q$ with some set $F$ in the family $\mc F$.  If the array is an $(N; t, k, 2)$-covering array, the corresponding family is said to be $(k,t)$-universal.

The study of such families dates back to the early 1970's, when one of the few exact results in this field was obtained by R\'enyi~\cite{Renyi} and Katona~\cite{Katona}, and independently by Kleitman and Spencer~\cite{KlSp}.  In our terminology, they showed
\[ \CAN(2, k, 2) = \min \left\{ N : k \le \binom{N - 1}{\floor{ \frac{N}{2}} - 1} \right\} = \lg k + \left( \frac12 + o(1) \right) \lg \lg k. \]
For larger $t$, Kleitman and Spencer~\cite{KlSp} showed there are constants $c_1, c_2 > 0$ such that
\[ c_1 2^t \lg k \le \CAN(t,k,2) \le c_2 t 2^t \lg k. \]

These results have since been extended to covering arrays over larger sets of symbols.  In the strength-two case, Gargano, K\"orner and Vaccaro~\cite{GKV93} established the asymptotic result
\[ \CAN(2,k,v) = \left( \frac12 + o(1) \right) v \lg k. \]

For higher strengths, there is a considerable gap between the best-known upper and lower bounds on $\CAN(t,k,v)$.  A general upper bound was given by Godbole, Skipper and Sunley~\cite{GSS96}, who studied when the uniformly random array $A \in S^{N \times k}$ is an $(N; t,k,v)$-covering array.

\begin{thm}[Godbole--Skipper--Sunley~\cite{GSS96}, 1996] \label{thm:classicbound}
For fixed integers $t, v \ge 2$, as $k \rightarrow \infty$,
\[ \CAN(t,k,v) \le \left(1 + o(1) \right) \frac{(t-1) \lg k}{\lg \frac{v^t}{v^t - 1}}. \]
\end{thm}

One can prove a lower bound by induction on $t$, reducing the problem to the $t=2$ case.  This shows $\CAN(t,k,v) \ge \left( \frac12 + o(1) \right) v^{t-1} \lg ( k - t + 2 )$.  In particular, when $t$ and $v$ are fixed, $\CAN(t,k,v) = \Theta_{t,v} \left( \lg k \right)$, but the correct dependence on $t$ and $v$ is unknown.  For convenience, we often consider the behaviour of the function in the limit as $k$ goes to infinity.

\begin{defi}
Given integers $t, v \ge 2$, we define
\[ d(t,v) = \limsup_{k \rightarrow \infty} \frac{\CAN(t,k,v)}{\lg k}. \]
\end{defi}

The result of Gargano, K\"orner and Vaccaro~\cite{GKV93} thus shows that $d(2,v) = \frac{v}{2}$, while Theorem~\ref{thm:classicbound} can be restated as saying $d(t,v) \le \frac{t-1}{\lg \frac{v^t}{v^t - 1}}$.  Note that in the limit as $v^t \rightarrow \infty$, this bound is asymptotically $(t-1)v^t \ln 2$.  

There have since been some improvements to the lower-order terms of this upper bound.  Franceti\'c and Stevens~\cite{FS16} showed that for all $t, v \ge 2$,
\[ d(t,v) \le \frac{v(t-1)}{\lg \frac{v^{t-1}}{v^{t-1} - 1}}. \]
Sarkar and Colbourn~\cite{SC16} gave an alternative proof of this improved bound, and further reduced the bound whenever $v$ is a prime power.  In this case, they proved
\[ d(t,v) \le \frac{v(v-1)(t-1)}{\log \frac{v^{t-1}}{v^{t-1} - v + 1}}. \]
While these bounds give the same asymptotics as Theorem~\ref{thm:classicbound} as $v^t$ tends to infinity, they are significantly better for small values of $v$ and $t$.  For further details of the best-known constructions for particular values of $v$ or $t$, we refer the reader to the excellent surveys of Lawrence, Kacker, Lei, Kuhn and Forbes~\cite{LKLKF11} in the binary setting and of Colbourn~\cite{Col04} in the general setting. 

\subsection{Our results} \label{subsec:results}

We begin by studying what is, in some sense, an inverse problem.  Rather than seeking the smallest array that covers all $t$-sets of columns, we fix the size of the array and try to maximise the number of covered $t$-sets.  This gives rise to the following extremal function.

\begin{defi}[Maximum coverage function]
Suppose we have a set $S$ of $v$ symbols, $N$ rows, $k$ columns and a strength $t \in \mathbb{N}$.  For an $N \times k$ array $A \in S^{N \times k}$, let $\Cov(A)$ denote the collection of all subsets of columns that are covered by $A$, and define $\covt(A)$ to be the number of sets of size $t$ in $\Cov(A)$.

We define the \emph{maximum coverage function, $\covm(N; t, k, v)$,} to be the maximum number of $t$-subsets that can be covered by such an array.  That is,
\[ \covm(N; t, k, v) = \max \left\{ \covt(A) : A \in S^{N \times k} \right\}. \]
\end{defi}

We note that similar notions have appeared previously in the literature.  Hartman and Raskin~\cite{HR04} and Maximoff, Trela, Kuhn and Kacker~\cite{MTKK10} suggested comparable lines of study, with a focus on developing heuristics for building small arrays that cover many sets.  On the other hand, our interest is in proving general bounds on the function $\covm(N; t, k, v)$.  More recently, Sarkar, Colbourn, De Bonis and Vaccaro~\cite{SCDV16} proved bounds on the sizes of almost-covering arrays, which are arrays that cover almost all $t$-sets.  This is more closely related to our investigation, and we shall discuss their results in our concluding remarks.

Note that in order for an array $A \in S^{N \times k}$ to cover even a single $t$-set $Q$ of columns, we must have at least $v^t$ rows, as each of the $v^t$ sequences in $S^t$ must appear as rows in $A_Q$.  This trivially gives $\covm(N; t, k, v) = 0$ for all $N \le v^t - 1$.  The first problem of interest is thus to determine $\covm(v^t; t, k, v)$, and this is the case on which we focus.

Given that the upper bound of Theorem~\ref{thm:classicbound} comes from the random array, this is a natural candidate to consider for our problem as well.  Let $A_{\mathrm{rand},v^t}$ be a uniformly random array chosen from $S^{v^t \times k}$.  Observe that for any subset $Q$ of $t$ columns, the $v^t$ rows of $(A_{\mathrm{rand},v^t})_Q$ are independent and uniformly distributed over $S^t$.  The probability that these rows are all distinct, and hence that $Q$ is covered, is thus
\[ \mathbb{P} \left( Q \in \Cov(A_{\mathrm{rand},v^t}) \right) = \frac{(v^t)!}{(v^t)^{v^t}}. \]
As $v^t$ tends to infinity, this is $e^{-(1 - o(1))v^t}$, and so in expectation $A_{\mathrm{rand},v^t}$ only covers an exponentially small fraction of all $t$-sets of columns.

A moment's thought reveals that this is far from optimal.  Indeed, consider the following block construction, where for simplicity we suppose that $k$ is divisible by $t$.  Partition the $k$ columns into $t$ equal-sized subsets $B_1, B_2, \hdots, B_t$, and build an array $A_{\mathrm{block}}$ whose rows are all sequences in $S^k$ that are constant on the sets $B_i$, $1 \le i \le t$.  We clearly have exactly $v^t$ rows, and a set $Q$ of columns is covered if and only if $\card{Q \cap B_i} \le 1$ for all $i$.  Therefore, when $k$ and $t$ are suitably large, we have
\[ \covt(A_{\mathrm{block}}) = \left( \frac{k}{t} \right)^t \sim \sqrt{2 \pi t} \cdot e^{-t} \binom{k}{t}. \]
This already gives a significantly better lower bound on $\covm(v^t; t, k, v)$ than the random array $A_{\mathrm{rand},v^t}$, but, as it turns out, is still far from the truth.  In our primary result below we give an upper bound on $\covm(v^t; t, k, v)$ and, whenever $v$ is a prime power, demonstrate its tightness by means of an explicit construction.

\begin{thm} \label{main}
Given integers $t, v \ge 2$, define
\[ c_{t,v} = \prod_{i=0}^{t-1} \frac{v^t - v^i}{v^t - 1}. \]
Then, for any $k \ge t$, we have the bound
\[ \covm(v^t; t, k, v) \le c_{t,v} \frac{k^t}{t!}. \]
If $v$ is a prime power, and $\frac{v^t-1}{v-1}$ divides $k$, we have equality; that is, $\covm(v^t; t, k, v) = c_{t,v} \frac{k^t}{t!}$.  For prime power $v$ and all $k \ge t$, we have
\[ \covm(v^t; t, k, v) \ge c_{t,v} \binom{k}{t}. \]
\end{thm}

Observe that
\[ c_{t,v} = \prod_{i=0}^{t-1} \frac{v^t - v^i}{v^t - 1} \ge \prod_{i=1}^{t-1} \left(1 - v^{-i} \right) \ge \left( 1 - v^{-1} \right) \left( 1 - \sum_{i=2}^{\infty} v^{-i} \right) = 1 - \frac{v+1}{v^2}. \]
This shows, perhaps surprisingly, that as soon as we have a large enough array to cover a single $t$-set of columns, one can already cover a large proportion of all such $t$-sets.  We can then use these very efficient arrays to build small covering arrays.

\begin{cor} \label{cor:CAN}
For integers $k \ge t \ge 2$ and a prime power $v \ge 2$, we have
\[ \CAN(t, k, v) \le \ceil{\frac{(t-1)\lg k + \lg (et)}{\lg \frac{1}{1 - c_{t,v}}}} v^t. \]
In particular, this implies that whenever $v$ is a prime power,
\[ d(t,v) \le \frac{(t-1)v^t}{\lg \frac{1}{1 - c_{t,v}}} \le \frac{(t-1)v^t}{2 \lg v - \lg (v+1)}. \]
Moreover, there is some absolute constant $v_0 \in \mathbb{N}$ such that for all integers $v \ge v_0$, we have
\[ d(t,v) \le \frac{(t-1) e^{t v^{-0.474}} v^t}{2 \lg v - \lg (v+1)}. \]
\end{cor}

This represents the first asymptotic improvement on the bound in Theorem~\ref{thm:classicbound}.  Recall that, as $v^t$ grows, the bounds from Theorem~\ref{thm:classicbound} and the subsequent improvements are all asymptotically $(t-1)v^t \ln 2$, and hence our bound is smaller by a factor of approximately $\ln v$ (provided that either $v$ is large and $t = o(v^{0.474})$, or $v$ is a prime power).  More precise calculations suggest that for prime power $v$, the bounds in Corollary~\ref{cor:CAN} improve the existing bounds whenever $t \ge 3$ and $v \ge 4$.

\subsection{Organisation and notation} \label{subsec:organisation}

The remainder of this paper is organised as follows.  In Section~\ref{sec:upperbound} we prove the upper bound of Theorem~\ref{main}, and in Section~\ref{sec:lowerbound} we establish the lower bound by providing an optimal construction.  Section~\ref{sec:covering} is devoted to the construction of small covering arrays and the proof of Corollary~\ref{cor:CAN}.  We then provide some concluding remarks and open problems in Section~\ref{sec:conclusion}.

We use standard combinatorial notation throughout this paper.  In particular, $[k]$ denotes the first $k$ positive integers, $\{ 1, 2, \hdots, k\}$.  Given a set $X$ and an integer $t$, $\binom{X}{t}$ is the collection of all $t$-subsets of $X$.  As defined previously, for an array $A$ and a subset $Q$ of its columns, $A_Q$ denotes the subarray of $A$ containing only the columns in $Q$.  We denote by $\Cov(A)$ the collection of all subsets of columns that are covered by $A$, and by $\covt(A)$ the number of sets of size $t$ in $\Cov(A)$.  Finally, we use $\lg$ for the binary logarithm, and $\ln$ for the natural logarithm.

\section{A general upper bound} \label{sec:upperbound}

We shall start our proof of Theorem~\ref{main} by proving the general upper bound. In order to do this, we shall show that the number of $t$-sets that can be covered by an array of size $v^t\times k$ is bounded by the Lagrangian of an auxiliary hypergraph, which we shall then bound.  First we present some useful preliminaries concerning Lagrangians in general.

\subsection{Lagrangians of $t$-uniform hypergraphs} \label{subsec:lagrangians}

Lagrangians, first introduced by Motzkin and Straus~\cite{MS} to give a proof of Tur\'an's theorem, have proven to be very useful in the study of extremal combinatorics.  Roughly speaking, the Lagrangian of a hypergraph determines the maximum possible density of a blow-up of the hypergraph.  We now define the Lagrangian more precisely.

\begin{defi}
Let $\mc H \subseteq {U \choose t}$ be a $t$-uniform hypergraph on a finite set $U$ of vertices. We say that a function $x:U\rightarrow \mathbb{R}$ is a legal weighting if $x(u)\geq 0$ for every $u\in U$ and $\displaystyle\sum_{u\in U} x(u)=1$. The weight polynomial of $\mc H$ evaluated at this weighting is given by
\[w(x,\mc H)=\sum_{e\in \mc H}\prod_{u\in e}x(u).\] 
The Lagrangian of $\mc H$ is defined to be $\lambda(\mc H)=\max w(x,\mc H)$, where the maximum is taken over all legal weightings $x$ of $\mc H$.
\end{defi}

We call a legal weighting $x$ optimal if $w(x, \mc H) = \lambda(\mc H)$. The following lemma, proven by Frankl and R\"odl~\cite{FR}, gives some information about the minimal supports of optimal weightings.

\begin{lemma}\label{F-R}
Let $\mc H$ be a $t$-uniform hypergraph on the vertex set $U$, and suppose $x$ is an optimal weighting where the number of vertices with non-zero weight is minimal. If $u,w\in U$ are vertices of non-zero weight, then there is an edge in $\mc H$ containing both $u$ and $w$.
\end{lemma}

\begin{proof}
Suppose for contradiction that vertices $u$ and $w$ have positive weight, but there is no edge of $\mc H$ containing both of them.  We define a parametrised function $x_{\eps}: U \rightarrow \mathbb{R}$, where
\[ x_{\eps}(s) = \begin{cases}
	x(u) + \eps &\textrm{if } s = u, \\
	x(w) - \eps & \textrm{if } s = w, \\
	x(s) & \textrm{otherwise.}
	\end{cases} \]
Observe that $x_{\eps}$ is a legal weighting whenever $\eps \in [-x(u), x(w)]$, and that at the boundaries either $x_{\eps}(u)$ or $x_{\eps}(w)$ becomes zero.  Moreover, since there is no edge containing both $u$ and $w$, $w(x_{\eps}, \mc H)$ is linear in $\eps$, and hence is maximised by some $\eps^* \in \{-x(u), x(w)\}$.  However, this gives a contradiction, as $x_{\eps^*}$ is then an optimal weighting with fewer non-zero weights.  Thus $u$ and $w$ must appear together in some edge of $\mc H$.
\end{proof}

The next lemma shows that we can bound the Lagrangian $\lambda( \mc H)$ of a $t$-uniform hypergraph $\mc H$ in terms of the Lagrangians of smaller $(t-1)$-uniform hypergraphs.  Given a vertex $u \in U$, the \emph{link hypergraph} $\mc H(u)$ is a $(t-1)$-uniform hypergraph on the vertex set $U \setminus \{ u \}$, with edges $e' \in \mc H(u)$ whenever $e' \cup \{ u \} \in \mc H$.

\begin{lemma}\label{Lagrange link}
Given $t \ge 2$, a $t$-uniform hypergraph $\mc H$ on the vertex set $U$ and a legal weighting $x$, we have
\[w(x,\mc H)\leq \frac{1}{t}\sum_{u\in U}x(u)(1-x(u))^{t-1}\lambda(\mc H(u)).\]
\end{lemma}

\begin{proof}
Note that if $x(u) = 1$ for any $u \in U$, then $w(x, \mc H) = 0$, and the inequality trivially holds.  Hence we may assume $x(u) < 1$ for all $u \in U$.  Since every edge of $\mc H$ has exactly $t$ vertices, double-counting gives
\[ w(x, \mc H) = \sum_{e \in \mc H} \prod_{u \in e} x(u) = \frac{1}{t} \sum_{u \in U} \sum_{u \in e \in \mc H} \prod_{w \in e} x(w) = \frac{1}{t} \sum_{u\in U} x(u) \sum_{e' \in \mc H(u)} \prod_{w \in e'} x(w). \]
As $x$ is a legal weighting, for every vertex $u$, $x$ must distribute a total weight of $1 - x(u)$ on the vertices in $U \setminus \{ u \}$.  Thus we can rescale the weights to obtain a legal weighting for the link hypergraph $\mc H(u)$ by defining $x_u(w) = \frac{x(w)}{1 - x(u)}$ for all vertices $w \in U \setminus \{ u \}$.  We then have
\begin{align*}
	w(x, \mc H) &= \frac{1}{t} \sum_{u\in U} x(u) \sum_{e' \in \mc H(u)} \prod_{w \in e'} x(w) = \frac{1}{t} \sum_{u \in U} x(u) (1 - x(u))^{t-1} \sum_{e' \in \mc H(u)} \prod_{w \in e'} x_u(w) \\
	&= \frac{1}{t} \sum_{u \in U} x(u) (1 - x(u))^{t-1} w(x_u, \mc H(u)) \le \frac{1}{t} \sum_{u \in U} x(u) (1 - x(u))^{t-1} \lambda (\mc H(u)),
\end{align*}
where the inequality follows from the definition of the Lagrangian.
\end{proof}

The final lemma of this subsection concerns the Lagrangian of the $v$-fold tensor product of a hypergraph with itself.  Given $t$-uniform hypergraphs, $\mc H_1,\hdots, \mc H_v$ on vertex sets $U_1,\hdots,U_v$ respectively, their tensor product, denoted $\otimes_{\ell=1}^v\mc H_\ell$, is a $t$-uniform hypergraph on the vertex set $\prod_{\ell=1}^v U_{\ell}$ with edges
\[\otimes_{\ell=1}^v\mc H_\ell = \left\{\{(u_{1,1},\hdots,u_{v,1}),\hdots,(u_{1,t},\hdots,u_{v,t})\} : \forall \ell \in [v], \; \{u_{\ell,1},\hdots,u_{\ell,t}\}\in \mc H_\ell \right\}.\]
Note that every $t$-tuple of edges $(e_1,\hdots,e_v) \in \prod_{\ell=1}^v \mc H_\ell$ gives rise to $(t!)^{v-1}$ edges in $\otimes_{\ell=1}^v\mc H_\ell$, as the vertices can be combined in every possible order.  We write $\mc H^{\otimes v}$ for the hypergraph obtained by taking a tensor product of $v$ copies of a hypergraph $\mc H$.

\begin{lemma}\label{Lagrange tensor}
For every $t$-uniform hypergraph $\mc H$,
\[\lambda(\mc H^{\otimes v})=\lambda(\mc H).\]
\end{lemma}

\begin{proof}
Let $U$ be the vertex set of $\mc H$.  First consider an optimal weighting $x$ for $\mc H$, and define a legal weighting $\widehat{x}$ on $U^v$ as follows:
\[ \widehat{x}((u_1,u_2,\hdots,u_v)) = \begin{cases} x(u_1) & \textrm{if } u_1 = u_2 = \hdots =u_v, \\ 0 & \textrm{otherwise.} \end{cases} \]
Now for every edge $\{u_1,\hdots,u_t\}$ in $\mc H$, the edge $\{(u_1,\hdots,u_1),\hdots,(u_t,\hdots,u_t)\}$ appears in $\mc H^{\otimes v}$ with the same weight, and so
\[ \lambda(\mc H^{\otimes v})\geq w(\widehat{x},\mc H^{\otimes v})\geq w(x,\mc H)=\lambda(\mc H). \]

For the reverse inequality, let $\widehat{x}$ be an optimal weighting for $\mc H^{\otimes v}$, and define the legal weighting $x$ on $U$ by $x(u) = \sum_{u_2,\hdots,u_v \in U} \widehat{x}((u,u_2,\hdots,u_v))$.  We then have
\begin{align*}
	\lambda(\mc H) &\ge w(x, \mc H) = \sum_{\{u_1, \hdots, u_t \} \in \mc H} \prod_{i=1}^t x(u_i) = \sum_{\substack{\{u_1,\hdots, u_t \} \in \mc H \\ u_{j,i} \in U, \; 2 \le j \le v, \; i\in[t]}} \prod_{i=1}^t \widehat{x}((u_i, u_{2,i},\hdots,u_{v,i}))
\\
&\ge \sum_{\substack{\{u_1, \hdots, u_t \} \in \mc H \\ \{ u_{j,1}, \hdots, u_{j,t}\} \in \mc H,\; 2\leq j\leq v}} \prod_{i=1}^t \widehat{x}((u_i, u_{2,i},\hdots,u_{v,i}))
\\
&= \sum_{\{ (u_1,u_{2,1},\hdots,u_{v,1}), \hdots, (u_t,u_{2,t},\hdots,u_{v,t}) \} \in \mc H^{\otimes v}} \prod_{i=1}^t \widehat{x}((u_i,u_{2,i},\hdots,u_{v,i}))= w(\widehat{x}, \mc H^{\otimes v}) = \lambda(\mc H^{\otimes v}),
\end{align*}
and thus $\lambda(\mc H) = \lambda(\mc H^{\otimes v})$.
\end{proof}

\subsection{The upper bound} \label{subsec:upperproof}

With these preliminaries in place, we may proceed with the proof of the upper bound in Theorem~\ref{main}.  As mentioned earlier, we shall bound the number of $t$-sets that can be covered by the Lagrangian of an auxiliary hypergraph $\mc H_{t,v}$, which we now introduce.

The vertex set of $\mc H_{t,v}$ is $[v]^{v^t}$ and  $t$ vectors $\vec{y}_1, \hdots, \vec{y}_t\in [v]^{v^t} $ form an edge in $\mc H_{t,v}$ whenever all $v^t$ vectors in $[v]^t$ appear as rows of the $v^t \times t$ matrix whose columns are $\vec{y}_1, \hdots, \vec{y}_t$.  Note that this condition implies the vectors $\vec{y}_i$ are pairwise-distinct, so this is indeed a $t$-uniform hypergraph.

For a $v^t\times k$ array $A$, whose entries we shall assume to belong to $[v]$, and $\vec{y}\in [v]^{v^t}$, define $B_{\vec{y}} = \{ a \in [k] : A_{\{a\}} = \vec{y} \}$ and note that the sets $B_{\vec{y}}$ partition the set $[k]$ of columns of $A$.  Observe that a $t$-set $Q \subseteq [k]$ can only be covered by $A$ if all elements of $Q$ belong to different parts $B_{\vec{y}}$, as otherwise two of the columns in $A_Q$ will be identical.  However, not all such $t$-sets are covered. A $t$-set $Q \subseteq [k]$ is covered by $A$ if and only if the columns of $A_Q$ form an edge of $\mc H_{t,v}$.

We can hence count the number of covered $t$-sets, finding
\[ \covt(A) = \sum_{e \in \mc H_{t,v}} \prod_{\vec{y} \in e} \card{B_{\vec{y}}}. \]
Since the sets $\left\{ B_{\vec{y}} : \vec{y} \in [v]^{v^t} \right\}$ partition $[k]$, the function $x : [v]^{v^t} \rightarrow \mathbb{R}$ given by $x(\vec{y}) = \frac{1}{k} \card{B_{\vec{y}}}$ is a legal weighting of the vertices of $\mc H_{t,v}$.  Hence
\[ \covt(A) = \sum_{e \in \mc H_{t,v}} \prod_{\vec{y} \in e} \card{B_{\vec{y}}} = k^t \sum_{e \in \mc H_{t,v}} \prod_{\vec{y} \in e} x(\vec{y}) = k^t w(x, \mc H_{t,v}) \le k^t \lambda( \mc H_{t,v}). \]
Now note that the hypergraph $\mc H_{t,v}$ is in fact independent of the array $A$ (which only determines the weighting of the vertices), and hence $k^t \lambda( \mc H_{t,v})$ bounds the number of $t$-sets that can be covered by any array of size $v^t\times k$.  The following proposition therefore gives the desired upper bound for Theorem~\ref{main}.

\begin{prop}\label{upperbound}
For all $t \ge 1$,
\[ \lambda(\mc H_{t,v}) \le \frac{c_{t,v}}{t!}=\frac{1}{t!}\prod_{i=0}^{t-1}\frac{v^t-v^{i}}{v^t-1}. \]
\end{prop}

We first establish a few simple lemmas that we shall use when proving Proposition~\ref{upperbound}.  The first shows that there are optimal weightings of $\mc H_{t,v}$ that are supported on a relatively small number of vertices.

\begin{lemma}\label{nonzero weights}
There is an optimal weighting $x$ of $\mc H_{t,v}$ for which
\[ \card{\mathrm{supp}(x)} = \card{ \left\{ \vec{y} \in [v]^{v^t} : x(\vec{y})\neq 0\right\}} \leq \frac{v^t-1}{v-1}.\]
\end{lemma}

\begin{proof}
Let $x$ be an optimal weighting of $\mc H_{t,v}$ minimising the number of vectors with non-zero weight. By Lemma~\ref{F-R}, if $\vec{y}$ and $\vec{z}$ are vectors with non-zero weight, then there must be an edge $e \in \mc H_{t,v}$ containing both $\vec{y}$ and $\vec{z}$.  Since every vector in $[v]^t$ appears as a row in the matrix whose columns are the vectors in $e$, it follows that for each choice of $a, b \in [v]$, there are exactly $v^{t-2}$ coordinates $i$ where $y_i = a$ and $z_i = b$.  In particular, the vectors of non-zero weight form the columns of an $(v^t;2,k,v)$-orthogonal array, and so by Theorem~\ref{thm:OAstrengthtwo} their number is bounded from above by $\frac{v^t-1}{v-1}$.
\end{proof}

The next lemma describes the link hypergraphs of $\mc H_{t,v}$.

\begin{lemma}\label{linkprod}
For any vertex $\vec{y}\in [v]^{v^t}$ in $\mc H_{t,v}$ of positive degree, $\mc H_{t,v}(\vec y)\cong \mc H_{t-1,v}^{\otimes v}$.
\end{lemma}

\begin{proof}
Recall that we have an edge $\{ \vec{y}_1, \hdots, \vec{y}_{t-1}, \vec{y} \} \in \mc H_{t,v}$ if and only if the $v^t \times t$ matrix $M$ formed with these column vectors has all vectors in $[v]^t$ as row vectors.  In particular, $\vec{y}$ must have $v^{t-1}$ entries equal to $a$ for every $a\in [v]$. Given $a\in[v]$, denote by $M^{(a)}$ the $v^{t-1}\times (t-1)$ matrix formed by taking those rows of $M$ that end with $a$, and then deleting the last (all-$a$) column. Since the rows of $M$ contain every vector of $[v]^t$ ending in $a$, it follows that the rows of $M^{(a)}$ consist of all vectors in $[v]^{t-1}$. In particular, the columns of $M^{(a)}$ form an edge in $\mc H_{t-1,v}$.

In other words, after a possible reordering of the rows of $M$, for $1 \le i \le t-1$ the vector $\vec{y}_i = ( \vec{w}_{1,i}, \vec{w}_{2,i},\hdots,\vec{w}_{v,i})^T$ should be the concatenation of vectors $\vec{w}_{j,i}\in [v]^{v^{t-1}}$, $j\in [v]$, such that for $1\leq j\leq v$ we have $\{ \vec{w}_{j,1}, \hdots, \vec{w}_{j,t-1} \}\in \mc H_{t-1,v}$.  This correspondence between the edges $\{ \vec{y}_1, \hdots, \vec{y}_{t-1} \} \in \mc H_{t,v}(\vec{y})$ and $\{ (\vec{w}_{1,1},\hdots,\vec{w}_{v,1}), \hdots, (\vec{w}_{1,t-1},\hdots,\vec{w}_{v,t-1}) \} \in \mc H_{t-1,v}^{\otimes v}$ gives the desired isomorphism between $\mc H_{t,v}(\vec{y})$ and $\mc H_{t-1,v}^{\otimes v}$.
\end{proof}

The final lemma solves an optimisation problem that shall appear in our proof of Proposition~\ref{upperbound}.

\begin{lemma}\label{maxima}
For $K \geq t\geq 2$, let $f(x_1,\hdots,x_K)=\sum_{i=1}^K x_i(1-x_i)^{t-1}$. The maximum of $f$, subject to $x_i \geq 0$ for every $1\leq i\leq K$ and $\sum_i x_i = 1$, is $\left(1-\frac{1}{K}\right)^{t-1}$.
\end{lemma}

\begin{proof}
Note that $f$ is a continuous function and the constraints define a compact set, so the maximum is well defined.  By taking $x_i = \frac{1}{K}$ for all $i$, we find $f \left( \frac{1}{K}, \hdots, \frac{1}{K} \right) = \left(1 - \frac{1}{K} \right)^{t-1}$, and so we only need to prove the upper bound.

Set $g(x)=x(1-x)^{t-1}$, and observe that 
\[g'(x)=(1 - tx) (1-x)^{t-2} \quad \textrm{and} \quad g''(x)=(tx - 2)(t-1)(1-x)^{t-3}.\]
Hence $g(x)$ is monotone increasing on $[0,\frac{1}{t}]$, monotone decreasing on $[\frac{1}{t},1]$, concave on $[0,\frac{2}{t}]$ and convex on $[\frac{2}{t},1]$.

Now let $(x_1,\hdots,x_K)$ maximise $f$, and suppose there was some index $i_0$ such that $x_{i_0} > \frac{1}{t}$.  Since $K \ge t$ and $\sum_i x_i = 1$, the average of the weights $x_i$ is at most $\frac{1}{t}$, and hence there must be some index $j_0$ such that $x_{j_0} < \frac{1}{t}$.  Put $\eps=\min \left\{\frac{1}{t}-x_{j_0},x_{i_0}-\frac{1}{t}\right\} > 0$.  If we replace $x_{j_0}$ by $x_{j_0}+\eps$ and $x_{i_0}$ by $x_{i_0}-\eps$, the monotonicity properties of $g$ imply the value of $f$ would increase, contradicting the fact that $(x_1,\hdots,x_K)$ maximises $f$.

Hence $x_i \in \left[0, \frac{1}{t} \right]$ for $1 \le i \le K$.  As $g$ is concave on this interval, Jensen's inequality gives
\[f(x_1,\hdots,x_K)=\sum_{i=1}^K x_i(1-x_i)^{t-1}= \sum_{i=1}^K g(x_i) \le K g\left( \frac{1}{K} \sum_{i=1}^K x_i \right) =\left(1-\frac{1}{K}\right)^{t-1}. \qedhere\]
\end{proof}

We are now in position to bound the Lagrangian of the hypergraph $\mc H_{t,v}$, thereby completing the proof of the upper bound from Theorem~\ref{main}.

\begin{proof}[Proof of Proposition~\ref{upperbound}]

For fixed $v$, we shall prove $\lambda(\mc H_{t,v}) \le \frac{c_{t,v}}{t!}$ by induction on $t$.
                                                                       
For the base case $t=1$, the hypergraph $\mc H_{v,1}$ is very simple. We have $v^v$ vertices corresponding to the vectors in $[v]^v$.  The edges of $\mc H_{v,1}$ are the singletons corresponding to vectors containing every $a\in [v]$. We thus have $v!$ edges, and the weight polynomial is simply the sum of the weights of the corresponding $v!$ vertices, whose maximum value is trivially at most $1$, which is equal to $\frac{c_{1,v}}{1!}$.

For the induction step, suppose $\lambda(\mc H_{t-1,v}) \le \frac{c_{t-1,v}}{(t-1)!}$ and consider $\mc H_{t,v}$.  Let $x$ be an optimal weighting of $\mc H_{t,v}$ with minimal support.  Suppose $\vec{y}_1, \hdots, \vec{y}_K$ are the vertices of $\mc H_{t,v}$ with non-zero weight, and let $x_1, \hdots, x_K$ represent their respective weights.  By Lemma~\ref{nonzero weights}, $K \le \frac{v^t - 1}{v-1}$.

Using Lemma~\ref{Lagrange link}, we find
\[ \lambda( \mc H_{t,v} ) = w(x, \mc H_{t,v}) \le \frac{1}{t} \sum_{\vec{y} \in [v]^{v^t}} x(\vec{y}) (1 - x(\vec{y}))^{t-1} \lambda( \mc H_{t,v}(\vec{y})) = \frac{1}{t} \sum_{i=1}^K x_i (1 - x_i)^{t-1} \lambda( \mc H_{t,v}(\vec{y}_i)). \]

By Lemma~\ref{linkprod}, $\mc H_{t,v}(\vec{y}_i) \cong \mc H_{t-1,v}^{\otimes v}$ for all $i$, and hence Lemma~\ref{Lagrange tensor} and the induction hypothesis give $\lambda( \mc H_{t,v}(\vec{y}_i)) = \lambda( \mc H_{t-1,v}) \le \frac{c_{t-1,v}}{(t-1)!}$ for all $i$.  Thus
\[ \lambda( \mc H_{t,v} ) \le \frac{c_{t-1,v}}{t!} \sum_{i=1}^K x_i (1 - x_i)^{t-1}. \]

Since the support must span at least one edge of $\mc H_{t,v}$, we have $K \ge t \ge 2$.  Moreover, since $x$ is a legal weighting, $x_i \ge 0$ for all $i$, and $\sum_i x_i = 1$.  Hence we may apply Lemma~\ref{maxima} to deduce that $\lambda(\mc H_{t,v}) \le \frac{c_{t-1,v}}{t!} \left( 1 - \frac{1}{K} \right)^{t-1}$.  As $K \le \frac{v^t - 1}{v-1}$, this can be further bounded by $\lambda( \mc H_{t,v} ) \le \frac{c_{t-1,v}}{t!} \left( 1 - \frac{v-1}{v^t - 1} \right)^{t-1}$.  Substituting in the definition of $c_{t-1,v}$, we obtain
\begin{align*}
	\lambda( \mc H_{t,v} ) &\le \frac{c_{t-1,v}}{t!} \left(1 - \frac{v-1}{v^t - 1} \right)^{t-1} = \frac{1}{t!} \left(\prod_{i=0}^{t-2} \frac{v^{t-1}-v^i}{v^{t-1} - 1} \right) \left( 1 - \frac{v-1}{v^t - 1} \right)^{t-1} \\
	&= \frac{1}{t!} \left(\prod_{i=0}^{t-2} \frac{v^{t-1}-v^i}{v^{t-1} - 1} \right)\left( \frac{v^t - v}{v^t - 1} \right)^{t-1}= \frac{1}{t!} \prod_{i=0}^{t-2} \frac{\left(v^{t-1}-v^i\right) \left(v^t - v \right)}{\left( v^{t-1} - 1\right) \left(v^t - 1 \right)} \\
	&= \frac{1}{t!} \prod_{i=0}^{t-2} \frac{v^{t}-v^{i+1}}{v^{t} - 1} =\frac{1}{t!} \prod_{i=1}^{t-1} \frac{v^{t}-v^i}{v^{t} - 1} = \frac{1}{t!} \prod_{i=0}^{t-1} \frac{v^t - v^i}{v^t - 1} = \frac{c_{t,v}}{t!},
\end{align*}
completing the proof.
\end{proof}

\section{An optimal construction} \label{sec:lowerbound}

In this section we prove the lower bounds of Theorem~\ref{main} by providing, for every prime power $v$, a linear algebraic construction of an array of size $v^t\times k$ that covers a large number of $t$-sets.

We begin by handling the case $k=\frac{v^t-1}{v-1}$. Let $\mathbb{F}_v$ be the $v$-element field, and consider the vector space $\mathbb{F}_v^t$.  Note that there are exactly $\frac{v^t - 1}{v-1}$ $1$-dimensional subspaces $L_1, \hdots, L_k$, and for each such subspace $L_i$, fix some non-zero vector $\vec{z_i} \in L_i \le \mathbb{F}_v^t$.  Now let $A_{\text{opt}}$ be a $v^t\times \frac{v^t-1}{v-1}$ array whose rows are indexed by the $v^t$ vectors in $\mathbb{F}_v^t$ and whose columns are indexed by $[k]$.  Given $\vec{y} \in \mathbb{F}_v^t$ and $i \in [k]$, we define the $(\vec{y},i)$ entry of $A_{\text{opt}}$ to be the scalar product $\vec{y}\cdot\vec{z}_i$.

We claim that a $t$-set $Q \subseteq [k]$ is covered by $A_{\text{opt}}$ if and only if the corresponding vectors $\{ \vec{z}_i : i \in Q \}$ are linearly independent. To see this, let $M$ be the $t\times t$ matrix containing the vectors $\{ \vec{z}_i : i \in Q \}$ as columns. For any $\vec{y} \in \mathbb{F}_v^t$, the corresponding row of $(A_{\text{opt}})_Q$ is $\vec{y}^T M$. Now the set $Q$ is covered by $A_{\text{opt}}$ if every possible vector occurs in this way, or, equivalently, if any vector from $\mathbb{F}_v^t$ can be obtained as $\vec{y}^T M$ for some appropriate vector $\vec{y}\in \mathbb{F}_v^t$. This happens precisely when the matrix $M$ is invertible, i.e. when the column vectors are linearly independent.

Now how many linearly independent sets $\{ \vec{z}_{j_1}, \hdots, \vec{z}_{j_t} \}$ are there?  If we have already chosen $i \ge 0$ linearly independent vectors $\vec{z}_{j_1},\hdots,\vec{z}_{j_i}$, the next vector $\vec{z}_{j_{i+1}}$, and hence the corresponding subspace $L_{j_{i+1}}$ it belongs to, cannot be in the $i$-dimensional subspace spanned by $\{\vec{z}_{j_1}, \hdots, \vec{z}_{j_i} \}$.  This forbids $\frac{v^i - 1}{v-1}$ of the possible vectors, leaving $\frac{v^t-1}{v-1}-\frac{v^i-1}{v-1}$ choices for $\vec{z}_{j_{i+1}}$. As the order in which the vectors are chosen does not matter, this gives a total of 
\[ \frac{1}{t!}\prod_{i=0}^{t-1}\left(\frac{v^t-1}{v-1}-\frac{v^i-1}{v-1}\right)=\frac{1}{t!}\prod_{i=0}^{t-1}\frac{v^t-v^i}{v-1}=\left(\frac{v^t-1}{v-1}\right)^{t}\frac{c_{t,v}}{t!} =c_{t,v} \frac{k^t}{t!}\]
different linearly independent sets of size $t$, and hence this is also the number of $t$-sets covered by $A_{\text{opt}}$. Note that this exactly matches the upper bound from Section~\ref{sec:upperbound}, which implies that in Proposition~\ref{upperbound}, we in fact determine the Lagrangian of the hypergraph $\mc H_{t,v}$ precisely for all prime powers $v$.

If $k$ is divisible by $\frac{v^t-1}{v-1}$, then we can take a blow-up of this linear algebraic construction, similar to the block construction in Section~\ref{sec:intro}. Partition the set $[k]$ of column indices into $\frac{v^t-1}{v - 1}$ parts of equal size, so that we have parts $\left\{ B_i : i \in \left[ \frac{v^t - 1}{v - 1} \right] \right\}$ with $\card{B_i} =\frac{k(v-1)}{v^t - 1}$ for all $i$.  We can now define the blown-up array $A_{\mathrm{opt}, \mathrm{block}}$ of size $v^t\times k$, where for $j \in B_i$, the $j$th column of $A_{\mathrm{opt}, \mathrm{block}}$ is the $i$th column of $A_{\mathrm{opt}}$.

It is easy to see that a set $Q \subseteq [k]$ is covered by $A_{\mathrm{opt}, \mathrm{block}}$ if and only if it contains at most one element from each block and the corresponding set of block indices is covered by $A_{\mathrm{opt}}$.  Thus
\[ \Cov (A_{\mathrm{opt}, \mathrm{block}}) = \bigcup_{Q' \in \Cov ( A_{\mathrm{opt}} )} \; \prod_{i \in Q'} B_i. \]

Restricting to sets of size $t$, we find
\begin{align*}
\covt( A_{\mathrm{opt}, \mathrm{block}} ) &= \sum_{\{ j_1, \hdots, j_t \} \in \Cov( A_{\mathrm{opt}} ) } \prod_{i=1}^t \card{B_{j_i} } \\
&=  \sum_{\{ j_1, \hdots, j_t \} \in \Cov( A_{\mathrm{opt}} ) } \left(\frac{k(v-1)}{v^t - 1} \right)^t = \left(\frac{k(v-1)}{v^t - 1}\right)^t \covt( A_{\mathrm{opt}} ) \\
&= \left(\frac{k(v-1)}{v^t - 1}\right)^{t} \left(\frac{v^t-1}{v-1}\right)^{t}\frac{c_{t,v}}{t!}=c_{t,v} \frac{k^t}{t!},
\end{align*}
showing we do have equality in Theorem~\ref{main} for this case.

For general $k$, as noted by Alon (personal communication), one can take a random partition of the set $[k]$ of column indices into the $\frac{v^t-1}{v-1}$ parts $\left\{ B_i : i \in \left[\frac{v^t - 1}{v-1} \right] \right\}$.  More precisely, for each $j \in [k]$ choose some $i \in \left[ \frac{v^t-1}{v-1} \right]$ independently and uniformly at random, and add $j$ to $B_i$.  We can now take the corresponding block construction $A$ based on $A_{\text{opt}}$, where for each $j \in B_i$, the $j$th column of $A$ is the $i$th column of $A_{\mathrm{opt}}$.  We again have that a set $Q$ is covered by $A$ if and only if it contains at most one element from each block and the corresponding set of block indices is covered by $A_{\mathrm{opt}}$. Accordingly, for a fixed $t$-set $Q$,
\[\mathbb{P}(Q \in \Cov(A))=\frac{t!\covt(A_{\text{opt}})}{\left(\frac{v^t-1}{v-1}\right)^t}=c_{t,v},\]
and so by linearity of expectation the expected number of $t$-sets covered by $A$ is $c_{t,v}\binom{k}{t}$. Hence there must be some block partition for which the number of covered $t$-sets is at least $c_{t,v}\binom{k}{t}$.  This concludes the proof of Theorem~\ref{main}.

To close this section, we note that in the special case $v = 2$, the columns appearing in $A_{\mathrm{opt}}$ are, up to an affine translation, the non-constant columns of the $2^t \times 2^t$ Hadamard matrix constructed by Sylvester~\cite{Sylvester}.  While this may not be immediately apparent from the recursive definition of these matrices, it follows from an equivalent formulation given in~\cite{Encycl}.  Here, the rows and columns of the matrix are indexed by vectors from $\mathbb{F}_2^t$, and the entry in the row corresponding to $\vec{y}$ and the column corresponding to $\vec{z}$ is $(-1)^{\vec{y} \cdot \vec{z}}$.

\section{Small covering arrays} \label{sec:covering}

In this section we shall build small covering arrays, thereby proving Corollary~\ref{cor:CAN}.  Roughly speaking, the idea is to combine many random copies of the construction from Section~\ref{sec:lowerbound} into a single array.  Since a $t$-set $Q$ is covered with large probability by each individual copy of the small construction, it follows that the probability that $Q$ is not covered by any of the copies will be exponentially small.  In order to show that we obtain a covering array, we need to prove that with positive probability, all of the $\binom{k}{t}$ $t$-sets of columns are covered.  This will follow from the Lov\'asz Local Lemma, first proven by Erd\H{o}s and Lov\'asz~\cite{EL75} and subsequently sharpened by Spencer~\cite{Spe77}.

\begin{thm}[Lov\'asz Local Lemma, 1975] \label{thm:LLL}
Let $\mc E_1, \mc E_2, \hdots, \mc E_m$ be events in some probability space.  Suppose there are $p \in [0,1]$ and $d \in \mathbb{N}$ such that for each $i \in [m]$, $\mathbb{P}(\mc E_i) \le p$, and the event $\mc E_i$ is mutually independent of a set of all but at most $d$ of the other events.  If $ep(d+1) \le 1$, then $\mathbb{P}\left( \cap_{i=1}^m \overline{\mc E_i} \right) > 0$.
\end{thm}

We can now proceed with the proof of Corollary~\ref{cor:CAN}.

\begin{proof}[Proof of Corollary~\ref{cor:CAN}]
We first handle the prime power case.  Suppose we have integers $k \ge t \ge 2$, and let $v \ge 2$ be a prime power.  We wish to build an $(rv^t; t, k, v)$-covering array $A$, where $r = \ceil{ \frac{(t-1) \lg k + \lg (et)}{\lg \frac{1}{1 - c_{t,v}}}}$.

For $1 \le \ell \le r$, let $A^{(\ell)}$ be an independent copy of our construction from Section~\ref{sec:lowerbound}.  That is, for each $j \in [k]$, let $i_{\ell, j} \in \left[ \frac{v^t -1 }{v-1} \right]$ be chosen independently and uniformly at random, and take the $j$th column of $A^{(\ell)}$ to be $(A_{\mathrm{opt}})_{\{i_{\ell,j}\}}$.   We then take $A$ to be the concatenation of $A^{(1)}, A^{(2)}, \hdots, A^{(r)}$, giving an $r v^t \times k$ array.  We wish to show that with positive probability $A$ is an $(rv^t; t, k, v)$-covering array.

Hence, for each $t$-subset $Q$ of the columns, we let $\mc E_Q$ be the event that $Q$ is \emph{not} covered by $A$.  Since the array $A$ contains each array $A^{(\ell)}$, $\ell \in [r]$, if $Q$ is not covered by $A$, then it is not covered by any $A^{(\ell)}$.  From Section~\ref{sec:lowerbound}, we have $\mathbb{P}\left(Q \notin \Cov(A^{(\ell)})\right) = 1 - c_{t,v}$, and by construction these subarrays are independent of one another.  We may therefore define $p$ by
\[ \mathbb{P}(\mc E_Q) = \mathbb{P}(Q \notin \Cov(A)) \le \mathbb{P}\left( \cap_{\ell \in [r]} \left\{ Q \notin \Cov\left(A^{(\ell)}\right) \right\} \right) = (1 - c_{t,v})^r = p. \]

Our construction of the array $A$ also ensures that the different columns are independent of one another.  In particular, this implies that the event $\mc E_Q$ is mutually independent of the set of all events that depend on a disjoint set of columns, i.e. $\{ \mc E_{Q'}: Q' \cap Q = \emptyset \}$.  For a fixed set $Q$, if $Q'$ intersects $Q$, it contains one of the $t$ columns of $Q$, and then there are fewer than $\binom{k}{t-1}$ choices for the remaining columns of $Q'$.  Hence we may take $d < t \binom{k}{t-1}$.

We thus have
\[ ep(d+1) \le et \binom{k}{t-1} (1 - c_{t,v})^r \le et k^{t-1}(1 - c_{t,v})^r \le 1,\]
where the final inequality follows from our choice of $r$.  Hence, by Theorem~\ref{thm:LLL}, with positive probability none of the events $\mc E_Q$ occur, which implies the existence of some such $rv^t \times k$ array $A$ that covers all of its $\binom{k}{t}$ $t$-subsets of columns.  This gives the claimed bound,
\[ \CAN(t,k,v) \le r v^t = \ceil{ \frac{(t-1)\lg k + \lg (et)}{\lg \frac{1}{1 - c_{t,v}}}} v^t. \]

For the next bound, we divide the above expression by $\lg k$ and take the limit as $k$ tends to infinity, giving
\[ d(t,v) \le \frac{(t-1)v^t}{\lg \frac{1}{1 - c_{t,v}}} \le \frac{(t-1)v^t}{2 \lg v - \lg (v+1)}, \]
where the second inequality follows from the bound $c_{t,v} \ge 1 - \frac{v+1}{v^2}$, derived below the statement of Theorem~\ref{main}.

Finally, we turn to the case when $v$ is not a prime power.  Here we use the trivial observation that for any $v \le v'$, $\CAN(t, k, v) \le \CAN(t, k, v')$, since projecting an array from a large set of symbols to a smaller set cannot cause any subset of columns to become uncovered.  In particular, it follows that $d(t, v) \le d(t,v')$.

Given $v$, let $q$ be the smallest prime power that is at least $v$.  Baker, Harman and Pintz~\cite{BHP01} proved that, provided $v$ is sufficiently large, $v \le q \le v + v^{0.526}$.  We thus have
\[ d(t,v) \le d(t,q) \le \frac{(t-1)q^t}{2 \lg q - \lg (q+1)} \le \frac{ (t-1) (1 + v^{-0.474})^t v^t}{ 2 \lg v - \lg (v+1)} \le \frac{ (t-1) e^{tv^{-0.474}} v^t}{2 \lg v - \lg(v + 1)},\]
as required.
\end{proof}

One may contrast our methods with those that have been used in previous constructions.  In~\cite{SC16} and~\cite{SCDV16}, small covering arrays were built by algebraically extending a random array.  Here, we do the opposite, taking random copies of a small linear algebraic array.  Since our initial array was exceedingly efficient in covering $t$-sets, we were able to reduce the number of random copies required, thus resulting in a smaller construction.  Indeed, the upper bound of Section~\ref{sec:upperbound} shows that this is the best $v^t \times k$ array one can start with.  However, to further improve the upper bound, one could perhaps find better ways to combine these copies, or perhaps start with a larger structured construction.

\section{Concluding remarks}\label{sec:conclusion}

In this paper we combined linear algebraic and probabilistic arguments to construct small covering arrays, asymptotically improving the upper bounds on $\CAN(t,k,v)$ by a factor of $\ln v$.  This involved the study of the extremal function $\covm(N; t,k,v)$, and we showed that at the lower threshold $N = v^t$ (the minimum size of an array that permits covering a single $t$-set) one can already cover a large proportion of all $t$-sets.  We close with some final remarks and possible directions for further research.

\paragraph{Almost-covering arrays}

As mentioned in the introduction, Sarkar, Colbourn, De Bonis and Vaccaro~\cite{SCDV16} studied (in greater generality) almost-covering arrays.  Given $N$ rows, $k$ columns, a strength $t$, a set $S$ of $v$ symbols, and a coverage fraction $\eps \in [0,1]$, an array $A \in S^{N \times k}$ is an $(N;t,k,v,\eps)$-almost-covering array if $A$ covers all but at most $\eps \binom{k}{t}$ $t$-subsets $Q$.  This relaxes the concept of a covering array, and one can again define an extremal function $\ACAN(t, k, v, \eps)$, which is the smallest $N$ for which an $(N; t,k,v,\eps)$-almost-covering array exists.

This is the inverse function of the $\covm(N; t,k,v)$ function; rather than fixing the size of the array and maximising the number of $t$-sets covered, we fix the number of sets to be covered, and minimise the size of the array needed.  In~\cite{SCDV16} it is proven that almost-covering arrays can be significantly smaller than covering arrays; indeed, the number of rows need not grow with $k$.  More precisely, they showed $\ACAN(t,k,v,\eps) \le v^t \ln \left( \frac{v^{t-1}}{\eps} \right)$, and that when $v$ is a prime power, the bound can be improved to $\ACAN(t,k,v,\eps) \le v^t \ln \left( \frac{2v^{t-2}}{\eps} \right) + v$.

As in Section~\ref{sec:covering}, we can improve the bounds by a $\lg v$ factor by concatenating random copies of the linear algebraic array $A_{\mathrm{opt}}$ until the expected number of uncovered sets is at most $\eps \binom{k}{t}$.  For prime power $v$, this gives $\ACAN(t,k,v,\eps) \le v^t \ceil{ \frac{\ln \frac{1}{\eps}}{2 \ln v - \ln (v + 1)}}$.  In particular, $A_{opt}$ itself shows $\ACAN(t,k,v,\eps) = v^t$ for $\eps \ge \frac{v+1}{v^2}$.  These results can again be extended to all large values of $v$ by replacing $v$ with the next prime power, and this gives effective bounds whenever $t = o(v^{0.474})$.

\paragraph{Explicit constructions of covering arrays}

To improve the bound on $\CAN(t,k,v)$, we consider a random construction and apply the Lov\'asz Local Lemma, making our result an existential one.  However, we hope that one could apply the algorithmic version of the Local Lemma, due to Moser and Tardos~\cite{MT10}, in our setting, which would perhaps result in an efficient Las Vegas algorithm to produce small covering arrays.  For a simpler analysis, one could take a slightly larger number of random copies of $A_{\mathrm{opt}}$, which would then give a covering array with high probability, resulting in an efficient Monte Carlo algorithm instead.

Regardless, in light of the many important applications of covering arrays, a great deal of interest lies in the explicit construction of efficient covering arrays.  In the binary setting, with $v = 2$, where randomised constructions give arrays of size $\mc O \left( t 2^t \lg k \right)$, Alon~\cite{Alon} gave an explicit construction of size $2^{\mc O ( t^4)} \lg k$.  Following a sequence of incremental improvements, Naor, Schulman and Srinivasan~\cite{NSchS} provided near-optimal arrays of size $t^{\mc O(\lg t)} 2^t \lg k$.

Given the inherent symmetry of the linear algebraic array $A_{\mathrm{opt}}$, it might be possible to deterministically concatenate copies of this array to form a small explicitly-constructed covering array.  Indeed, this might even give better upper bounds on $\CAN(t,k,v)$ than those we obtained in Section~\ref{sec:covering} using probabilistic means.  Taking this approach, the minimum number of copies of $A_{\mathrm{opt}}$ needed is given by the following hashing problem with an algebraic twist.

\begin{ques} \label{ques:hashing}
Given $k \ge t \ge 2$ and a prime power $v \ge 2$, what is the minimum $r$ such that there is a collection of $k$ sequences in $(\mathbb{F}_v^t)^r$ (that is, the sequences have length $r$, and each entry is a vector in $\mathbb{F}_v^t$) with the property that for every $t$-subset of the sequences, there is some coordinate where the $t$ corresponding vectors are linearly independent?
\end{ques}

In the standard hashing problem, one would only require the existence of a coordinate where the sequences were pairwise-distinct, whereas here we impose a stronger algebraic condition.  Given that the best-known bounds for the hashing problem come from probabilistic constructions (see~\cite{FK84}), it might be difficult to improve on the bounds from Section~\ref{sec:covering} this way.  However, deterministic solutions to Question~\ref{ques:hashing} would result in explicit constructions of (hopefully) small covering arrays.

\paragraph{Evolution of $\covm(N; t, k, v)$}

Another way to improve the construction in Section~\ref{sec:covering} would be to replace $A_{\mathrm{opt}}$, and instead concatenate random copies of some other array.  The upper bound proven in Section~\ref{sec:upperbound} shows that we cannot hope to do better with arrays of size $v^t$, but it might be beneficial to consider initial arrays with a larger number of rows.  With this in mind, it would be of interest to determine how the function $\covm(N; t,k,v)$ grows as $N$ increases.  In particular, at the other extreme we observe that $\CAN(t,k,v) = \min \left\{ N : \covm(N; t,k,v) = \binom{k}{t} \right\}$, and so complete knowledge of $\covm(N; t,k,v)$ would solve the covering array problem as well.

On a much finer scale, what happens for $N = v^t + s$ for small values of $s$?  In $A_{\mathrm{opt}}$, if a $t$-subset $Q$ of columns is not covered, then the vectors $\vec{z}_i \in \mathbb{F}_v^t$ that the columns are mapped to form a matrix of rank at most $t-1$, and hence at most $v^{t-1}$ rows appear in the subarray $(A_{\mathrm{opt}})_Q$.  This implies that to increase the number of covered $t$-subsets, we need to add at least $v^t - v^{t-1}$ new rows, almost doubling the size of $A_{\mathrm{opt}}$.

Of course, there could be other arrays of size $v^t + s$ that do not contain $A_{\mathrm{opt}}$ as a subarray, but cover a larger number of $t$-sets.  Our proof of the optimality of $A_{\mathrm{opt}}$ when $N = v^t$ relied heavily on the fact that if a $t$-set is covered, then each sequence in $S^t$ appears exactly once as a row in the corresponding subarray, leading to useful linear algebraic interpretations of coverage.  This rigid structure is lost for larger arrays, rendering the analysis more difficult.  Still, we feel that without this structure, one should not be able to cover a larger number of $t$-sets.  To make this intuition more precise, we offer the following question.

\begin{ques} \label{ques:plusone}
When $v$ is a prime power, are $\covm(v^t + 1; t,k,v)$ and $\covm(v^t; t,k,v)$ equal?  What is the smallest $s$ for which $\covm(v^t + s; t, k, v) > \covm(v^t; t,k,v)$?
\end{ques}

\paragraph{Non-uniform coverage}

Finally, when dealing with covering arrays, we have only focussed on the number of subsets of some fixed size $t$ that are covered.  However, one could instead consider all covered sets, and look to maximise $\card{\Cov(A)}$ instead.

In the context of set shattering, similar questions have indeed been considered.  The famous Sauer--Shelah inequality states that a set family $\mc F$ must shatter at least $\card{\mc F}$ sets in total.  Given its numerous applications, a pressing open problem is the classification of all families that attain this bound with equality.  For details on this line of research, see, for example,~\cite{BR,MR14,RM}.

For our problem, we instead ask which families of a given size maximise the number of shattered sets.  Note that a family of size $m$ can shatter sets of size at most $\floor{\lg m}$, and, if $\lg m$ is small compared to the size $k$ of the ground set, then almost all such sets will have size precisely $\floor{\lg m}$.  If $m = 2^t$, our construction maximises the number of $t$-sets shattered, and hence one might expect it also maximises the total number of shattered sets.

\begin{ques}
Given $0 \le m \le 2^k$, which set family $\mc F \subseteq 2^{[k]}$ of size $\card{\mc F} = m$ maximises the number of shattered sets?
\end{ques}

\paragraph{Acknowledgements}  The authors would like to express their gratitude towards Noga Alon, Gal Kronenberg and Benny Sudakov for helpful discussions at various stages of this project.  Special thanks go to Ugo Vaccaro for bringing~\cite{SCDV16} to our attention, prompting the extension of our initial results on set shattering to general covering arrays.


\begin{thebibliography}{99}

\bibitem{Alon}
N.~Alon,
\newblock{\emph{Explicit construcion of exponential sized families of $k$-independent sets}},
\newblock{Discrete Math.} \textbf{58.2} (1986), 191--193.

\bibitem{BHP01}
R.~C.~Baker, G.~Harman and J.~Pintz,
\newblock{\emph{The difference between consecutive primes, II}},
\newblock{P. Lond. Math. Soc.} \textbf{83.3} (2001), 532--562.

\bibitem{BR}
B.~Bollob\'as and A.~J.~Radcliffe,
\newblock{\emph{Defect Sauer results}},
\newblock{J. Comb. Theory A} \textbf{72} (1995), 189--208.

\bibitem{Encycl}
P.~J.~Cameron (ed.),
\newblock{\emph{Hadamard matrices}},
\newblock{Encyclopaedia of Design Theory}, \url{http://www.maths.qmul.ac.uk/~leonard/designtheory.org/library/encyc/topics/had.pdf}.

\bibitem{Col04}
C.~J.~Colbourn,
\newblock{\emph{Combinatorial aspects of covering arrays}},
\newblock{Le Matematiche (Catania)} \textbf{58} (2004), 121--167.

\bibitem{EL75}
P.~Erd\H{o}s and L.~Lov\'asz,
\newblock{\emph{Problems and results on $3$-chromatic hypergraphs and some related questions}},
\newblock{Infinite and Finite Sets (A.~Hajnal, R.~Rado and V.~T.~S\'os, eds.)}, North-Holland, Amsterdam (1975), 609--628.

\bibitem{FS16}
N.~Franceti\'c and B.~Stevens,
\newblock{\emph{Asymptotic size of covering arrays: an application of entropy compression}},
\newblock{J. Comb. Des.} (2016), doi: 10.1002/jcd.21553.

\bibitem{FR}
P.~Frankl and V.~R\"odl,
\newblock{\emph{Hypergraphs do not jump}},
\newblock{Combinatorica} \textbf{4} (1984), 149--159.

\bibitem{FK84}
M.~L.~Fredman and J.~Koml\'os,
\newblock{\emph{On the size of separating systems and families of perfect hash functions}},
\newblock{SIAM J. Algebra. Discr.} \textbf{5.1} (1984), 61--68.

\bibitem{GKV93}
L.~Gargano, J.~K\"orner and U.~Vaccaro,
\newblock{\emph{Sperner capacities}},
\newblock{Graph. Combinator.} \textbf{9} (1993), 31--46.

\bibitem{GSS96}
A.~P.~Godbole, D.~E.~Skipper and R.~A.~Sunley,
\newblock{\emph{$t$-covering arrays: upper bounds and Poisson approximations}},
\newblock{Comb. Probab. Comput.} \textbf{5} (1996), 105--118.

\bibitem{HR04}
A.~Hartman and L.~Raskin,
\newblock{\emph{Problems and algorithms for covering arrays}},
\newblock{Disc. Math.} \textbf{284} (2004), 149--156.

\bibitem{HSS99}
A.~S.~Hedayat, N.~J.~A.~Sloane and J.~Stufken,
\newblock{\emph{Orthogonal Arrays: Theory and Applications}},
\newblock{Springer Science and Business Media} (1999).

\bibitem{Katona}
Gy.~O.~H.~Katona,
\newblock{\emph{Two applications (for search theory and truth functions) of Sperner type theorems}},
\newblock{Period. Math. Hungarica} \textbf{3.1-2} (1973), 19--26.

\bibitem{KlSp}
D.~J.~Kleitman and J.~Spencer,
\newblock{\emph{Families of $k$-independent sets}},
\newblock{Discrete Math.} \textbf{6.3} (1973), 255-262.

\bibitem{LKLKF11}
J.~Lawrence, R.~N.~Kacker, Y.~Lei, D.~R.~Kuhn and M.~Forbes,
\newblock{\emph{A survey of binary covering arrays}},
\newblock{Electron. J. Combin.} \textbf{18.1} (2011), P84.

\bibitem{MTKK10}
J.~R.~Maximoff, M.~D.~Trela, D.~R.~Kuhn and R.~Kacker,
\newblock{\emph{A method for analyzing system state-space coverage within a $t$-wise testing framework}},
\newblock{4th Annual IEEE Systems Conference} (2010), 598--603.

\bibitem{MR14}
T.~M\'esz\'aros and L.~R\'onyai,
\newblock{\emph{Shattering-extremal set systems of VC dimension at most $2$}},
\newblock{Electron. J. Combin.} \textbf{21.4} (2014) P4.30.

\bibitem{MT10}
R.~A.~Moser and G.~Tardos,
\newblock{\emph{A constructive proof of the general Lov\'asz Local Lemma}},
\newblock{J. ACM} \textbf{57.2} (2010), 11.

\bibitem{MS}
T.~S.~Motzkin and E.~G.~Straus,
\newblock{\emph{Maxima for graphs and a new proof of a theorem of Tur\'an}},
\newblock{Canad. J. Math.} \textbf{17} (1965), 533--540.

\bibitem{NSchS}
M.~Naor, L.~J.~Schulman and A.~Srinivasan,
\newblock{\emph{Splitters and near-optimal derandomization}},
\newblock{Foundations of Computer Science, Proceedings}, IEEE (1995), 182--191.

\bibitem{PB46}
R.~L.~Plackett and J.~P.~Burman,
\newblock{\emph{The design of optimum multifactorial experiments}},
\newblock{Biometrika} \textbf{33.4} (1946), 305--325.

\bibitem{Rao46}
C.~R.~Rao,
\newblock{\emph{Hypercubes of strength $d$ leading to confounded designs in factorial experiments}},
\newblock{Bull. Calcutta Math. Soc.} \textbf{38} (1946), 67--78.

\bibitem{Rao47}
C.~R.~Rao,
\newblock{\emph{Factorial experiments derivable from combinatorial arrangements of arrays}},
\newblock{J. Roy. Stat. Soc., Suppl.} \textbf{9} (1947), 128--139.

\bibitem{Renyi}
A.~R\'enyi,
\newblock{\emph{Foundations of Probability}},
\newblock{Wiley}, New York (1971).

\bibitem{RM}
L.~R\'onyai and T.~M\'esz\'aros,
\newblock{\emph{Some combinatorial applications of Gr\"obner bases}},
\newblock{International Conference on Algebraic Informatics}, Springer Berlin Heidelberg (2011), 65--83.

\bibitem{SC16}
K.~Sarkar and C.~J.~Colbourn,
\newblock{\emph{Upper bounds on the size of covering arrays}},
\newblock{arXiv:1603.07809} (2016).

\bibitem{SCDV16}
K.~Sarkar, C.~J.~Colbourn, A.~De~Bonis and U.~Vaccaro,
\newblock{\emph{Partial covering arrays: algorithms and asymptotics}},
\newblock{International Workshop on Combinatorial Algorithms}, Springer International Publishing (2016), 437--448.

\bibitem{Spe77}
J.~H.~Spencer,
\newblock{\emph{Asymptotic lower bounds for Ramsey functions}},
\newblock{Discrete Math.} \textbf{20} (1977), 69--76.

\bibitem{Sylvester}
J.~J.~Sylvester,
\newblock{\emph{Thoughts on inverse orthogonal matrices, simultaneous sign successions, and tessellated pavements in two or more colours, with applications to Newton's rule, ornamental tile-work, and the theory of numbers}},
\newblock{Philos. Mag.} \textbf{34} (1867), 461--475.

\end{thebibliography}
\end{document}